\documentclass{article}

\usepackage{amsmath}
\usepackage{amssymb}
\usepackage{amsthm}

\usepackage{times}

\newtheorem{lemma}{Lemma}[section]
\newtheorem{theorem}[lemma]{Theorem}
\newtheorem{definition}[lemma]{Definition}
\newtheorem{corollary}[lemma]{Corollary}
\newtheorem{example}{Example}

\author{Henry Towsner}
\title{Metastability and the Furstenberg-Zimmer Tower II: Polynomial and Multidimensional Szemer\'edi's Theorem}
\date{\today}

\begin{document}
\maketitle

\begin{abstract}    
  The Furstenberg-Zimmer structure theorem for $\mathbb{Z}^d$ actions says that every measure-preserving system can be decomposed into a tower of primitive extensions.  Furstenberg and Katznelson used this analysis to prove the multidimensional Szemer\'edi's theorem, and Bergelson and Liebman further generalized to a polynomial Szemer\'edi's theorem.  Beleznay and Foreman showed that, in general, this tower can have any countable height.  Here we show that these proofs do not require the full height of this tower; we define a weaker combinatorial property which is sufficient for these proofs, and show that it always holds at fairly low levels in the transfinite construction (specifically, $\omega^{\omega^{\omega^\omega}}$).
\end{abstract}

\section{Introduction}
If $\mathcal{X}$ is a measure preserving system acted on by a group $\Gamma$, the Furstenberg-Zimmer structure theorem shows that $\mathcal{X}$ may be decomposed into a tower of primitive extensions: that is, a tower $(\mathcal{Y}_\alpha)_{\alpha\leq\theta}$ such that $\mathcal{Y}_0$ is the trivial factor, limit levels are the limit of the preceeding factors, and each $\mathcal{Y}_{\alpha+1}$ is compact relative to some subgroup $\Delta\subseteq\Gamma$ and weak mixing relative to all $S\in\Gamma\setminus\Delta$ (see \cite{FurstenbergBook}).

When $\mathcal{X}$ is separable, a simple countable argument shows that this tower must have countable height, and Beleznay and Foreman \cite{foreman96} have shown that even when $\Gamma$ is $\mathbb{Z}$, the tower may reach any countable height.

The structure theorem is commonly used to prove finitary combinatorial results, and we might hope that, for these finitary applications, only a limited portion of the tower is necessary.  In \cite{avigad09}, Avigad and the author showed that the proof of Szemer\'edi's Theorem from the structure theorem for $\mathbb{Z}$ actions requires only the first $\omega^{\omega^\omega}$ levels of the tower.

In this paper, we apply similar methods to Bergelson and Liebman's multidimensional polynomial Szemer\'edi Theorem \cite{bergelson96}; as such, we follow the proof from \cite{bergelson96} closely, applying theorems from there directly without repeating the proof when possible.  These methods generalize Furstenberg and Katznelson's multidimensional Szemer\'edi Theorem \cite{furstenberg79}, and therefore apply to that argument as well.  In the one-dimensional polynomial case, the corresponding bound is $\omega^{\omega^{\omega^\omega}}$; the higher bound is due to the use of the PET induction scheme, which has order-type $\omega^\omega$.  In the multidimensional case, the situation is slightly more complicated, since we need a system of nested towers with each section having this larger height.  (In the multidimensional linear case, the theorems here show that the corresponding system of nested towers with each section having height $\omega^{\omega^\omega}$ suffices.)

A central idea in this paper is that, while the property of being a weakly mixing extension is quite infinitary, the finitary consequences of being weak mixing can also be extracted from a sufficiently long sequence of factors which are all ``almost'' weak mixing with the same parameters.  While there is no countable bound on how tall a tower must be to ensure that an extension is weak mixing, there is a countable bound which is sufficient to guarantee the existence of these almost weak mixing extensions.  This type of approximation to infinitary convergence has been called ``metastability'' by Tao \cite{tao08Norm}.

The bounds obtained here are not optimal; however they are, in some sense, ``a priori'': they are extracted directly from the proof, without additional combinatorial techniques.  \cite{avigad09} discusses the logical methods underlying this extraction.

Regarding optimal bounds, Furstenberg's original ergodic proof \cite{furstenberg77} used only $\omega$ levels of the tower to prove Szemer\'edi's Theorem; that method has not directly been generalized to the multidimensional or polynomial case, but more recent methods \cite{host05,tao06,ziegler07} are generally believed to be sufficient to show that $\omega$ levels of the tower suffice for the multidimensional polynomial Szemer\'edi's Theorem as well (this has not, to our knowledge, been written down, although \cite{ziegler08} can be seen as implying the claim for the single dimensional polynomial Szemer\'edi Theorem).

We are grateful to Vitaly Bergelson for answering questions about the proof of the polynomial Szemer\'edi Theorem, and to Jeremy Avigad, with whom most of the new techniques in this paper were originally developed.

\section{Preliminaries}
Bergelson and Liebman's polynomial Szemer\'edi Theorem~\cite{bergelson96} states:
\begin{theorem}
  Let $p_{1,1},\ldots,p_{1,t},p_{2,1},\ldots,p_{2,t},\ldots,p_{k,1},\ldots,p_{k,t}$ be a collection of polynomials with rational coefficients taking on integer values on the integers and satisfying $p_{i,j}(0)=0$ for all $i,j$.  Then for any $\delta>0$, there is an $N$ large enough so that if $S$ is any subset of $\mathbb{Z}^d$ with density at least $\delta$ and $v_1,\ldots,v_t\in\mathbb{Z}^d$, there exists an integer $n$ and a vector $u\in\mathbb{Z}^d$ so that
\[u+\sum_{j=1}^tp_{i,j}(n)v_j\in S\]
for every $i\leq k$.
\end{theorem}
Using the correspondence principle introduced by Furstenberg~\cite{furstenberg77}, they obtain this as a consequence of the following ergodic theorem:
\begin{theorem}
  Let $(X,\mathcal{B},\mu,\mathbb{Z}^d)$ be a dynamical system, let $T_1,\ldots,T_t$ actions of elements of $\mathbb{Z}^d$, and let $p_{1,1}(n), \ldots,$ $p_{1,t}(n),$ $p_{2,1}(n),$ $\ldots,$ $p_{2,t}(n),$ $\ldots,$ $p_{k,1}(n), \ldots,p_{k,t}(n)$ be a collection of polynomials with rational taking on integer values on the integers and satisfying $p_{i,j}(0)=0$ for all $i,j$.  Then for any $A\in\mathcal{B}$ with $\mu(A)>0$,
\[\liminf_{N\rightarrow\infty}\frac{1}{N}\sum_{n=0}^{N-1}\mu(\left(\prod_{j=1}^tT_j^{-p_{1,j}(n)}\right)A\cap\cdots\cap\left(\prod_{j=1}^tT_j^{-p_{k,j}(n)}\right)A)>0.\]
\end{theorem}

Fix the integer $d$.  Throughout this paper, we will be concerned with measure preserving systems of the form $(X,\mathcal{B},\mu,\Gamma)$ where $\Gamma$ is a subgroup of $\mathbb{Z}^d$.  We will also be concerned with minimal sets of generators for $\Gamma$; that is, linearly independent sequences $T_1,\ldots,T_t\in\Gamma$ generating $\Gamma$.

\begin{definition}
If $(X,\mathcal{B},\mu,\Gamma)$ is an extension of $(Y,\mathcal{C},\nu,\Gamma)$, we say it is a \emph{compact} extension if for every $f\in L^2(X)$ and any $\epsilon,\delta>0$, there exist $B\in\mathcal{C}$ with $\nu(B)>1-\epsilon$ and a finite set of functions $h_1,\ldots,h_K\in L^2(X)$ such that for each $R\in\Gamma$, $\min_{1\leq i\leq K}||R(f\cdot \chi_B)-h_l||_y<\delta$ for all $y\in B$.
\end{definition}
See \cite{FurstenbergBook} for an extensive discussion of the properties of compact extensions.  Note that every function in $L^2(Y)$ is compact relative to $Y$.

\begin{definition}
  Let $\mathcal{X}:=(X,\mathcal{B},\mu,\mathbb{Z}^d)$ be a dynamical system and let $\Gamma\subseteq\mathbb{Z}^d$ be a subgroup and $T_1,\ldots,T_t\in\Gamma$ a minimal set of generators.  For any $L^\infty$ function $g$, we define
\[H_g^{n,T_1,\ldots,T_t}:=\frac{1}{n^{k}}\sum_{\vec i\in [0,n]^{t}}(T^{-i_1} \cdots T^{-i_t}(g\otimes g)).\]
\end{definition}
The Mean Ergodic Theorem for $\mathbb{Z}^d$ actions \cite{wiener39} implies that the functions $H^{n,T_1,\ldots,T_t}_g$ converge to a limit $H^{T_1,\ldots,T_t}_g$.  In particular, it is standard that an extension generated by functions of the form $H_g^{T_1,\ldots,T_t}f$ is compact (relative to the group generated by $T_1,\ldots,T_t$), and that conversely, every compact function is a limit of such functions.

\begin{definition}
  An increasing tower of factors of height $\gamma$ is a sequence of factors $(\mathcal{Y}_\delta)_{\delta<\gamma}$ such that $\alpha<\beta<\gamma$ implies $\mathcal{Y}_\alpha\subseteq\mathcal{Y}_\beta$ and whenever $\lambda<\gamma$ is a limit ordinal, $\mathcal{Y}_\lambda$ is generated by $\bigcup_{\beta<\lambda}\mathcal{Y}_\beta$.

  If $\Gamma$ is a group, we define $Z_\Gamma(\mathcal{Y})$ to be the space of all functions compact relative to $\mathcal{Y}$ with respect to the group $\Gamma$.  Given a fixed action of $\mathbb{Z}^d$ on $\mathcal{Y}$, define $Z_t(\mathcal{Y})$ to be the space generated by the union of $Z_\Gamma(\mathcal{Y})$ as $\Gamma$ ranges over subgroups of $\mathbb{Z}^d$ of dimension $\geq t$.

  If $\Gamma$ is a group, a tower of $\Gamma$-compact extensions is an increasing tower of factors $(\mathcal{Y}_\delta)_{\alpha<\gamma}$ such that for each $\alpha$, $Z_\Gamma(\mathcal{Y}_\alpha)\subseteq\mathcal{Y}_{\alpha+1}$.

  If $\Gamma\subsetneq\mathbb{Z}^d$ is a group, a tower of $\Gamma^+$-compact extensions is an increasing tower of factors $(\mathcal{Y}_\delta)_{\alpha<\gamma}$ such that for each $\alpha$ and each $S\in\mathbb{Z}^d\setminus\Gamma$, $Z_{\Gamma\cup\{S\}}(\mathcal{Y}_\delta)\subseteq\mathcal{Y}_{\alpha+1}$.

%  We say $\mathcal{X}$ is a $\gamma$-\emph{almost-primitive} extension of $\mathcal{Y}$ if $\mathcal{X}$ is generated by by a countable collection of factors $\{\mathcal{X}_i\}_{i\in I}$ such that for each $i\in I$, $\mathbb{Z}^d=\Gamma^i_c\times\Gamma^i_w$, $\mathcal{X}_i$ is a compact extension of $\mathcal{Y}$ with respect to $\Gamma^i_c$, and either $\Gamma^i_w$ is the trivial group or $\mathcal{Y}$ is generated by an increasing tower $(\mathcal{Y}_\delta)_{\alpha<\gamma}$ such that for each $\alpha<\gamma$, every $f,g\in L^\infty(\mathcal{X})$, and every $\Gamma'\supsetneq\Gamma^i_c$, $H^{\Gamma'}_g\ast_{\mathcal{Y}_\alpha}f\in L^\infty(\mathcal{Y}_{\alpha+1})$.
\end{definition}
(Note that in a tower of $\Gamma$-compact extensions, we do not require that $\mathcal{Y}_{\alpha+1}$ be a compact extension of $\mathcal{Y}_\alpha$; rather, we require that it contain every compact extension of $\mathcal{Y}_\alpha$.)

%In particular, for any $\gamma$ we may define a canonical increasing tower $(\mathcal{Y}_\delta)_{\alpha<\gamma}$ so that $\mathcal{X}$ is a $\gamma$-almost-primitive extension of the factor generated by the union of the tower and each element of the tower is a $\gamma$-almost-primitive extension of the previous level:
\begin{definition}
Let $\mathcal{Z},\eta$ be given.  We define a tower of factors by main induction on $n\leq d$ and side induction on $\alpha$:
\begin{itemize}
\item $\mathcal{Y}^{d-n,\eta}_0(\mathcal{Z}):=\mathcal{Z}$
\item $\mathcal{Y}^{d,\eta}_{\alpha+1}(\mathcal{Z}):=Z_d(\mathcal{Y}^{d,\eta}_\alpha(\mathcal{Z}))$
\item $\mathcal{Y}^{d-n-1,\eta}_{\alpha+1}(\mathcal{Z}):=Z_{d-n-1}(\mathcal{Y}^{d-n-1,\eta}_\alpha(\mathcal{Z}))\cup\mathcal{Y}^{d-n,\eta}_\eta(\mathcal{Y}^{d-n-1,\eta}_\alpha(\mathcal{Z}))$\footnote{We could replace this with $\bigcup_{\gamma<\eta}Z_{d-n-1}(\mathcal{Y}^{d-n,\eta}_\gamma(\mathcal{Y}^{d-n-1,\eta}_\alpha(\mathcal{Z})))$ while making only small changes in the proofs in this paper; however the system $Z_{d-n-1}(\mathcal{Y}^{d-n,\eta}_\eta(\mathcal{Y}^{d-n-1,\eta}_\alpha(\mathcal{Z})))$ can contain elements which are not approximately weak mixing relative to $\mathcal{Y}^{d-n-1,\eta}_\alpha(\mathcal{Z})$.}
\item For limit $\lambda$, $\mathcal{Y}^{d-n,\eta}_\lambda:=\bigcup_{\beta<\lambda}\mathcal{Y}^{d-n,\eta}_\beta$
\end{itemize}
\end{definition}
In particular, note that for any $\mathcal{Z}$, $\eta$, $n$, $\{\mathcal{Y}^{d-n,\eta}_\alpha(\mathcal{Z})\}_{\alpha<\eta}$ is a tower of $\Gamma$-compact extensions for every $\Gamma\subseteq\mathbb{Z}^d$ of dimension $d-n$.  When we take $\mathcal{Z}$ to be the trivial factor, note that $\mathcal{Y}^{0,\eta}_1=\mathcal{X}$ for any $\eta$ (since everything in $\mathcal{X}$ is compact with respect to the group of dimension $0$).

The main result is:
\begin{theorem}
For every $\epsilon>0$, every $\Gamma\subseteq\mathbb{Z}^d$, every linearly independent $T_1,\ldots,T_t\in\mathbb{Z}^d\setminus\Gamma$, every tower $(\mathcal{Y}_\delta)_{\delta<\eta}$ of $\Gamma\cup\{S\}$-compact extensions for every $S$ generated by $T_1,\ldots,T_t$, every sequence of polynomials $p_{1,1},\ldots,p_{1,t},p_{2,1},\ldots,p_{2,t},\ldots,p_{k,1},\ldots,p_{k,t}$ as in the polynomial Szemer\'edi Theorem, all functions $f_1,\ldots,f_k$ in $L^\infty(\mathcal{X})$, there is a $\delta<\omega^{\omega^{\omega^\omega}}$ and an $n$ such that, for all $m\geq n$,
\[\frac{1}{m}\sum_{i<m}\int \left|E(\prod_{r\leq k}\prod_{j\leq t}T_j^{-p_{r,j}(i)}f_r\mid\mathcal{Y}_\delta)-\prod_{r\leq k}\prod_{j\leq t}T_j^{-p_{r,j}(i)}E(f_r\mid\mathcal{Y}_\delta)\right|d\mu<\epsilon.\]
\end{theorem}
Note that, in particular, the towers $\mathcal{Y}^{d-n,\eta}_\alpha(\mathcal{Z})$, where $\Gamma$ has dimension $d-n$, always satisfy the premise.  We obtain this inductively using a stronger property, namely that this holds not for one $\delta$, but for many $\delta$ simultaneously in the same $n$.

If $\theta$ and $\eta$ are ordinals, $(\theta,\eta]$ denotes the interval $\{\delta\mid\theta<\delta\leq\eta\}$.

\begin{definition}
  If $\alpha$ is an ordinal, $s$ is an \emph{$\alpha$-sequence} if $s=(s_\beta)_{\beta\leq\alpha}$ is a strictly increasing sequence of ordinals indexed by ordinals less than or equal to $\alpha$.  Say $t$ is a \emph{$\beta$-subsequence} of $s$ if $t$ is a $\beta$-sequence and a subsequence of $s$.

If $s$ is an $\alpha$-sequence and $P(\delta)$ is any property, say $P$ \emph{holds for $s$-many $\delta$} if for every $\beta<\alpha$, there is a $\delta$ in $(s_\beta,s_{\beta+1}]$ such that $P(\delta)$ holds.
\end{definition}

\section{Approximating Weak Mixing}
The following metastable form of the Mean Ergodic Theorem follows immediately from Theorem 4.4 of \cite{avigad09}:
\begin{theorem}
  Let $\epsilon>0$, $B>0$, and let $(\mathcal{Y}_\delta)$ be an increasing tower of factors.  Then there is a natural number $K$ such that for every $\alpha^K$-sequence $s$ and every $g$ in $L^\infty(\mathcal{X})$ with $||g||_\infty\leq B$, there are a natural number $n$ and an $\alpha$-subsequence $t$ of $s$ such that the property
  \begin{quote}
    for every $m\geq n$ and $h$ with $||h||_{L^2(\mathcal{X})}\leq B$, $||H^{m,T}_g\ast_{\mathcal{Y}_\delta}h-H^T_g\ast_{\mathcal{Y}_\delta}h||<\epsilon$
  \end{quote}
holds for $t$-many $\delta$.
\label{ergodic}
\end{theorem}

The following theorem is proven almost identically to the analogous Theorem 5.1 in \cite{avigad09}; the only difference is that we have to use the $\Gamma$-compactness of $g$ with respect to every element of the tower to obtain the needed $\Gamma\cup\{T\}$-compactness of $H^T_g\ast_{\mathcal{Y}_\delta}h_\delta$. 
\begin{theorem}
For every $\epsilon>0$ and $B>0$, there is a natural number $K$ such that for every $\Gamma\subseteq\mathbb{Z}^d$, every $T\in\mathbb{Z}^d\setminus\Gamma$, every tower $(\mathcal{Y}_\delta)_{\delta<\eta}$ of $\Gamma\cup\{T\}$-compact extensions, every $\alpha\geq\omega$, every $\alpha^K$-sequence $s$, and every $f,g$ compact with respect to $\Gamma$ relative to $\mathcal{Y}_0$ with $||f||_\infty,||g||_\infty\leq B$, there are an $n$ and an $\alpha$-subsequence $t$ of $s$ such that the property
  \begin{quote}
    for every $m\geq n$, $\frac{1}{m}\sum_{i<m}\int\left[E(fT^{-i}g\mid\mathcal{Y}_\delta)-E(f\mid\mathcal{Y}_\delta)T^{-i} E(g\mid\mathcal{Y}_\delta)\right]^2d\mu<\epsilon$\\
  \end{quote}
holds for $t$-many $\delta$.
\label{wm}
\end{theorem}
\begin{proof}
For any $\delta$, if we set $h_\delta$ equal to $f-E(f\mid \mathcal{Y}_\delta)$, we have
  \begin{eqnarray*}
    \frac{1}{m}\sum_{i<m}\int\left[E(fT^{-i}g\mid\mathcal{Y}_\delta)-E(f\mid\mathcal{Y}_\delta)T^{-i}E(g\mid\mathcal{Y}_\delta)\right]^2d\mu\\
    =\frac{1}{m}\sum_{i<m}\int\left[E((h_\delta+E(f\mid\mathcal{Y}_\delta))T^{-i}g\mid\mathcal{Y}_\delta)-E(f\mid\mathcal{Y}_\delta)T^{-i}E(g\mid\mathcal{Y}_\delta)\right]^2d\mu\\
    =\frac{1}{m}\sum_{i<m}\int\left[E(h_\delta T^{-i}g\mid\mathcal{Y}_\delta)\right]^2d\mu\\
    =\frac{1}{m}\sum_{i<m}\int E(h_\delta T^{-i}g\mid\mathcal{Y}_\delta)E(h_\delta T^{-}ig\mid\mathcal{Y}_\delta)d\mu\\
    =\int E(h_\delta \frac{1}{m}\sum_{i<m}T^{-}igE(h_\delta T^{-i}g\mid\mathcal{Y}_\delta)\mid\mathcal{Y}_\delta)d\mu\\
    =\int E(h_\delta H^{m,T}_g\ast_{\mathcal{Y}_\delta}h_\delta\mid\mathcal{Y}_\delta) d\mu\\
    =\int h_\delta \cdot H^{m,T}_g\ast_{\mathcal{Y}_\delta}h_\delta d\mu
  \end{eqnarray*}
Since $g$ is compact relative to $\mathcal{Y}_0$ with respect to $\Gamma$, it compact relative to $\mathcal{Y}_\delta$ with respect to $\Gamma$.  Then $H^T_g\ast_{\mathcal{Y}_\delta}h_\delta=\lim_{m\rightarrow\infty}\frac{1}{m}\sum_{i<m}T^ig E(h_\delta T^{-i}g\mid\mathcal{Y}_\delta)$ is a limit of a sum of a product of a function in $L^2(\mathcal{Y}_\delta)$ with a function compact relative to $\mathcal{Y}_\delta$, and is therefore itself compact relative to $\mathcal{Y}_\delta$ with respect to $\Gamma$.  Further, by construction it is compact relative to $\mathcal{Y}_\delta$ with respect to $T$.  It is standard (see \cite{FurstenbergBook}) that $H^T_g\ast_{\mathcal{Y}_\delta}h_\delta$ is therefore compact relative to $\mathcal{Y}_\delta$ with respect to the group generated by $\Gamma\cup\{T\}$, and therefore $H^T_g\ast_{\mathcal{Y}_\delta}h_\delta\in L^2(\mathcal{Y}_{\delta+1})$.
%Observe that $H_g\ast_{\mathcal{Y}_\delta}h_\delta$ is compact relative to $\Gamma,S$ with respect to $\mathcal{Y}_\delta$: by assumption, $g$ is compact relative to $\Gamma$ with respect to $\mathcal{Y}_\delta$, and $H_g\ast_{\mathcal{Y}_\delta}h_\delta$ belongs to the factor generated by $\mathcal{Y}_\delta$ and $g$, and is therefore compact relative to $\Gamma$.  Compactness relative to $S$ follows by Lemma \ref{compact_def}, and therefore compactness relative to $\Gamma,S$ follows by Lemma \ref{compact_conjoin}.  In particular, $\Gamma,S$ is a group of size $t+1$, so $H_g\ast_{\mathcal{Y}_\delta}h_\delta$ belongs to $\mathcal{Y}_{\delta+1}$.

Given $\epsilon>0$ and $B>0$, choose $K$ given by Lemma \ref{ergodic} for $\epsilon/2B$, $B$.  We claim that $2K$ satisfies the claim.  Suppose we are given an $\alpha^{2K}$-sequence $s$ and $f$ and $g$ with $||f||_\infty\leq B$, $||g||_\infty\leq B$.  Since $\alpha\geq\omega$, we may restrict $s$ to the initial $(\omega\cdot\alpha)^K$-subsequence, and by our choice of $K$, there is an $n$ and an $\omega\cdot\alpha$-subsequence $t$ with the property that
\begin{quote}
  for every $m\geq n$ and $h$ with $||h||_{L^2(\mathcal{X})}\leq B$, $||H^{m,T}_g\ast_{\mathcal{Y}_\delta}h-H_g\ast_{\mathcal{Y}_\delta}h||<\epsilon/2$ ($\ast$)
\end{quote}
holds for $t$-many $\delta$.  Let $t'$ be the $\alpha$-sequence obtained by setting $t'_\beta:=t_{\omega\cdot\beta}$ for each $\beta\leq\alpha$.  Then for each $\beta<\alpha$ and each $i$, there is a $\delta_i\in(t_{\omega\cdot\beta+i},t_{\omega\cdot\beta+i+1}]$ such that ($\ast$) holds.  In particular, there is some $i$ such that $||h_{\delta_i+1}-h_{\delta_i}||=||E(f\mid\mathcal{Y}_{\delta_i+1})-E(f\mid\mathcal{Y}_{\delta_i})||<\epsilon/2B^2$, and so for $\delta:=\delta_i$, we have
\begin{eqnarray*}
h_\delta\cdot(H^{m,T}_g\ast_{\mathcal{Y}_\delta}h_\delta)=&h_\delta\cdot((H^{m,T}_g\ast_{\mathcal{Y}_\delta}h_\delta)-(H^T_g\ast_{\mathcal{Y}_\delta}h_\delta))\\
&+(h_\delta-h_{\delta+1})\cdot(H^T_g\ast_{\mathcal{Y}_\delta}h_{\delta+1})\\
&+h_{\delta+1}\cdot(H^T_g\ast_{\mathcal{Y}_\delta}h_\delta).
\end{eqnarray*}
For every $m\geq n$, by ($\ast$), the first term is bounded by $||h_\delta||_\infty\cdot\epsilon/2B\leq\epsilon/2$ since $||h_\delta||_\infty\leq B$.  The second term is bounded in $L^2(\mathcal{X})$ norm by $(\epsilon/2B^2)\cdot||H^T_g\ast_{\mathcal{Y}_\delta}h_{\delta+1}||_\infty$, which is less than $\epsilon/2$ since $||H^T_g\ast_{\mathcal{Y}_\delta}h_{\delta+1}||_\infty\leq B^2$.  The integral of the last term is $0$ since $h_{\delta+1}$ is orthogonal to $\mathcal{Y}_{\delta+1}$ and $H^T_g\ast_{\mathcal{Y}_\delta}h_\delta$ is an element of $\mathcal{Y}_{\delta+1}$.  Hence $\int h_\delta\cdot (H^{m,T}_g\ast_{\mathcal{Y}_\delta}h_\delta)d\mu<\epsilon$ as required.
\end{proof}

\section{Polynomials}
The definitions in this section are essentially those of \cite{bergelson96}.
\begin{definition}
  An \emph{integral polynomial} is a polynomial with rational coefficients taking integer values on the integers.  An \emph{integral-zero polynomial} is a polynomial taking the value $0$ at $0$.
\end{definition}

Let $t$ be fixed.  If $\langle p_{j}\rangle_{j\in[1,t]}$ is a sequence of integral polynomials of degree at most $D$, the \emph{degree} of $\langle p_j\rangle$ is $\max_{i\in[1,t]} deg(p_i)$ and the \emph{weight}, $w(\langle p_j\rangle)$, is the pair $(r,d)$ such that whenever $i>r$, $deg (p_i)=0$, and $deg (p_r)=d\geq 1$.  We order weights by the lexicographic ordering, $(r,d)>(s,e)$ if $r>s$ or $r=s$ and $d>e$.

We say two such sequences $\langle p_j\rangle,\langle q_j\rangle$ are \emph{equivalent} if they have the same weight and the leading coefficients of the polynomials $p_r$, $q_r$ are the same.

The \emph{degree} of a finite collection $A:=\{\langle p_{i,j}\rangle_{j\in[1,t]}\}_{i\in [1,k]}$ of such sequences is the maximum of the degree of any of its elements.  The \emph{weight matrix}, $wm(A)$, is the matrix
\[\left(\begin{tabular}{ccc}
$N_{1,1}$&$\cdots$&$N_{1,D}$\\
$\vdots$&$\vdots$&$\vdots$\\
$N_{t,1}$&$\cdots$&$N_{t,D}$\\
\end{tabular}\right)\]
where $D$ is the degree of $A$ and $N_{s,d}$ is the number of equivalence classes with weight $(s,d)$ in the collection.

\begin{example}
  The system with $t=2$, $D=5$ given by
  \begin{eqnarray*}
&\{\{19n,0\},\{6n^2,0\},\{7n^2+19n,0\},\{7n^2,0\},\{4n^4,n^2\},
\{n^2,3n^3\},\{n^2,3n^3+2n\},\\&\{n,2n^3+3n\},\{10n^5,n^3+4n^2+4n\},\{0,n^3+2n\},\{n^5,n^3+n^2\}\}   
  \end{eqnarray*}
has weight matrix
\[\left(
  \begin{tabular}{ccccc}
1&2&0&0&0\\
0&1&3&0&0    
  \end{tabular}
\right).\]
\end{example}

We introduce an ordering on weight matrices (with $t$ and $D$ fixed): $N<M$ if for some $(r,d)$, $N_{(r',d')}=M_{(r',d')}$ for all $(r',d')>(r,d)$ and $N_{(r,d)}<M_{(r,d)}$.  It is easy to see that this is a well-ordering of order type $\omega^{t\cdot D}$, and when $M$ is a weight matrix, we will write $o(M)$ for the height of $A$ in this ordering (that is, $o(M)$ is the order-type of $\{N\mid N<M\}$).  If $A$ is a finite set of sequences of integral polynomials, we write $o(A):=o(wm(A))$.

\begin{definition}
  We say two sequences of polynomials $\langle p_i\rangle$, $\langle q_i\rangle$ are \emph{essentially distinct} if for some $i$, $p_i(n)-q_i(n)$ is not constant.
\end{definition}

Let $A$ be such a set of sequences.  For any $h$, we define a new set $\tilde A_h:=\{\langle p_j(n)\rangle, \langle p_j(n+h)-p_j(h)\rangle \mid \langle p_j(n)\rangle\in A\}$.  When $deg \langle p_j\rangle=1$, $p_j(n)=p_j(n+h)$ for every $j$.  But whenever $deg \langle p_j\rangle\geq 2$, there is at most one $h$ such that $p_j(n)=p_j(n+h)-p_j(h)$.  Note that $\tilde A_h$ has the same weight matrix as $A$.

Assume $A$, and therefore $\tilde A_h$, contains no polynomials of weight $<(1,1)$.  Let $\langle p_j\rangle\in \tilde A_h$ be chosen with minimal weight and define $A_h:=\{\langle p'_j(n)-p_j(n)\rangle\mid \langle p'_j\rangle\in \tilde A_h\}$.  Note that if $\langle p'_j\rangle\in \tilde A_h$ is not equivalent to $\langle p_j\rangle$ then $\langle p'_j-p_j\rangle$ is equivalent to $\langle p'_j\rangle$, and if $\langle p'_j\rangle$ is equivalent to $\langle p_j\rangle$ then $w(\langle p'_j-p_j\rangle)<w(\langle p'_j\rangle)$.  In particular, this means that the weight matrix of $A_h$ preceeds the weight matrix of $A$.

\begin{definition}
  If $\vec T:=T_1,\ldots,T_t$ is a sequence of transformations and $\langle p_j\rangle_{j\in[1,t]}$ is a sequence of polynomials, we write $\vec T^{\vec p}(n):=\prod_{j=1}^t T^{-p_{j}(n)}_j$.
\end{definition}

\begin{definition}
  If $\mathcal{X}$ is a dynamical system and $T_1,\ldots,T_t$ is a minimal generating set of a subgroup of $\mathbb{Z}^d$, we say $\mathcal{X}$ satisfies $T_1,\ldots,T_t,A$-SZP if for any $B$ with $\mu(B)>0$,
\[\liminf_{m\rightarrow\infty}\frac{1}{m}\sum_{i<m}\mu(\bigcap_{\vec p\in A}\vec T^{\vec p}(i)B)>0.\]

We say $\mathcal{X}$ is SZP if it is $T_1,\ldots,T_t,A$-SZP for every choice of $T_1,\ldots,T_t$ and every set $A$ of pairwise distinct sequences of length $t$ of integral-zero polynomials.
\end{definition}
Note that distinct integral-zero polynomials are essentially distinct.

\section{Approximating Weak Mixing Along Polynomials}
We recall the following technical results from \cite{avigad09}:
\begin{definition}
  A formula $\phi(\vec x,\delta)$ is \emph{continuous} in an ordinal parameter $\delta$ if whenever $\phi(\vec x,\gamma)$ for every $\gamma<\beta$, also $\phi(\vec x,\beta)$.
\end{definition}

\begin{lemma}
  Suppose $\phi_1(\vec x,\delta)$ and $\phi_2(\vec x,\delta)$ are continuous in $\delta$.  Fix $\vec x$.

Suppose that for each $i\in\{1,2\}$, for every $B>0$ there is a $\theta_i<\omega^{p_i}$ such that for every $\alpha^{\theta_i}$-sequence $s$ with $\alpha\geq\omega$ and every $f$ with $||f||_{L^\infty}\leq B$, there are a natural number $n_i$ and an $\alpha$-subsequence $t$ of $s$ such that the property
\begin{quote}
  for every $m\geq n_i$, $\phi_i(\vec x,\delta)$
\end{quote}
holds for $t$-many $\delta$.

Then for every $B>0$ there is a $\theta<\omega^{p_1+p_2-1}$ such that for every $\alpha^\theta$-sequence $s$ with $\alpha\geq\omega$ and every $f$ with $||f||_{L^\infty}\leq B$, there are a natural number $n$ and an $\alpha$-subsequence $t$ of $s$ such that the property
\begin{quote}
  for every $m\geq n$, $\phi_1(\vec x,\delta)$ and $\phi_2(\vec x,\delta)$
\end{quote}
holds for $t$-many $\delta$.
\label{tech:pair}
\end{lemma}

\begin{lemma}
Suppose there is a $\theta<\omega^p$ such that for every $\alpha^\theta$-sequence $s$ with $\alpha\geq\omega$ and every $f$ with $||f||_{L^\infty}\leq B$, there are a natural number $n$ and an $\alpha$-subsequence $t$ of $s$ such that the property
\begin{quote}
  for every $m\geq n$, $\phi(f,m,\delta)$
\end{quote}
holds for $t$-many $\delta$.

Suppose also that $\epsilon>0$ is such that whenever $||f-f'||_{L^2}<\epsilon$ and $\phi(f,m,\delta)$ holds, also $\phi'(f',m,\delta)$ holds.  Let $\phi$ be continuous in $\delta$.  Then there is a $\theta<\omega^{2p-1}$ such that for every $\alpha^\theta$-sequence $s$ with $\alpha\geq\omega$ and every $f$ with $||f||_{L^\infty}\leq B$, there are a natural number $n$ and an $\alpha$-subsequence $t$ of $s$ such that the property
\begin{quote}
  for every $m\geq n$, $\phi'(E(f\mid\mathcal{Y}_\delta),m,\delta)$ and $\phi'(f-E(f\mid\mathcal{Y}_\delta),m,\delta)$
\end{quote}
holds for $t$-many $\delta$.
  \label{tech:split}
\end{lemma}

\begin{lemma}
  Suppose there is a $\theta_0<\omega^p$ such that for every $\alpha^{\theta_0}$-sequence $s$ with $\alpha\geq\omega$ and every $f$ with $||f||_{L^\infty}\leq B$, there are a natural number $n_0$ and an $\alpha$-subsequence $t$ of $s$ such that the property
  \begin{quote}
    for every $m\geq n_0$, $\phi_0(f,m,\delta)$
  \end{quote}
holds for $t$-many $\delta$.

Suppose that, additionally, for every $d$ there is a $\theta_d<\omega^q$ such that for every $\alpha^{\theta_d}$-sequence $s$ with $\alpha\geq\omega$ and every $f$ with $||f||_{L^\infty}\leq B$, there is a natural number $n_d$ and an $\alpha$-subsequence $t$ of $s$ such that the property
\begin{quote}
  for every $m\geq n_d$, $\phi_d(f,m,\delta)$
\end{quote}
holds for $t$-many $\delta$.

If $\phi_i$ is continuous in $\delta$ for each $i$ then there is a $\theta<\omega^{p+q}$ such that for every $\alpha^\theta$-sequence $s$ with $\alpha\geq\omega$ and every $f$ with $||f||_{L^\infty}\leq B$, there are an $n$, an $N$, and an $\alpha$-subsequence $t$ of $s$ such that the property
\begin{quote}
  $\phi_0(f,N,\delta)$ and for every $m\geq n$, $\phi_N(f,m,\delta)$
\end{quote}
holds for $t$-many $\delta$.
\label{tech:tech}
\end{lemma}

Recall that if $\mathcal{X}$ is a measure-preserving system and $\mathcal{Y}$ is a factor, $\mathcal{X}\times_{\mathcal{Y}}\mathcal{X}$ is again a measure-preserving system with factor $\mathcal{Y}$.  $L^\infty(\mathcal{Y})$ can be identified as a subset of $L^\infty(\mathcal{X}\times_{\mathcal{Y}}\mathcal{X})$, and if $f$ and $g$ are elements of $L^\infty(\mathcal{X})$ then $f\otimes g$ is an element of $L^\infty(\mathcal{X}\times_{\mathcal{Y}}\mathcal{X})$.  Thus, the most basic elements of $L^\infty(\mathcal{X}\times_{\mathcal{Y}}\mathcal{X})$ can be viewed as tensor products of elements of $L^\infty(\mathcal{X})$.  We define the \emph{simple} elements of $L^\infty(\mathcal{X}\times_{\mathcal{Y}}\mathcal{X})$ to be those that can be represented as finite sums of such basic elements.  The advantage to focusing on simple elements is that if $f$ is such an element then $f$ can be viewed as an element of $L^\infty(\mathcal{X}\times_{\mathcal{Y}^{k,\eta}_\delta(\mathcal{Z})}\mathcal{X})$ for each $k,\eta,\mathcal{Z},\delta$ simultaneously.

More precisely, we define $L^\infty_0(\mathcal{X}\times\mathcal{X})$ to be the set of finite formal sums of such basic elements; then each element $f$ of $L^\infty_0(\mathcal{X}\times\mathcal{X})$ denotes an element of $L^\infty(\mathcal{X}\times_{\mathcal{Y}}\mathcal{X})$ for any $\mathcal{Y}$.  Note that if $f$ and $g$ are elements of $L^\infty_0(\mathcal{X}\times\mathcal{X})$ it makes sense to talk about $f+g$ and $E(f\mid\mathcal{Y})$ as elements of $L^\infty_0(\mathcal{X}\times_{\mathcal{Y}}\mathcal{X})$.  We may define an $L^\infty$ bound of such a formal sum in the natural way, taking $||\sum_{i<n}c_if_i\otimes g_i||_{\infty}=\sum_{i<n}c_i||f_i||_{\infty}||g_i||_{\infty}$.  Such a bound is an upper bound for the true $L^\infty$ bound in $L^\infty(\mathcal{X}\times_{\mathcal{Y}}\mathcal{X})$ and respects the usual properties of the $L^\infty$ norm with respect to sums and products.

Using this, we can generalize Theorem \ref{wm} to the relative square $\mathcal{X}\times_{\mathcal{Y}_\delta}\mathcal{X}$; we could go further, extending to the relative square of the relative square, and so on, but we will not need to do so here.
\begin{lemma}
For every $\epsilon>0$ and $B>0$, there is a natural number $K$ such that for every $\Gamma\subseteq\mathbb{Z}^d$, every $T\in\mathbb{Z}^d\setminus\Gamma$, every tower $(\mathcal{Y}_\delta)_{\delta<\eta}$ of $\Gamma\cup\{T\}$-compact extensions, every $\alpha\geq\omega$, every $\alpha^K$-sequence $s$, and every $f,g\in L_0^\infty(Z_\Gamma(\mathcal{Y}_0)\times Z_\Gamma(\mathcal{Y}_0))$ with $||f||_\infty,||g||_\infty\leq B$, there are an $n$ and an $\alpha$-subsequence $t$ of $s$ such that the property
  \begin{quote}
    for every $m\geq n$, $\frac{1}{m}\sum_{i<m}\int\left[E(f(T^{-i}\otimes T^{-i})g\mid\mathcal{Y}_\delta)-E(f\mid\mathcal{Y}_\delta)(T^{-i}\otimes T^{-i}) E(g\mid\mathcal{Y}_\delta)\right]^2d\mu<\epsilon$
  \end{quote}
holds for $t$-many $\delta$.
\label{wm2}
\end{lemma}
\begin{proof}
  By Lemma \ref{tech:split} and the subadditivity of the left hand side, it suffices to consider the cases where $f=f_1\otimes f_2$, $g=g_1\otimes g_2$, and either $E(f_i\mid\mathcal{Y}_\delta)=0$ or $E(f_i\mid\mathcal{Y}_\delta)=f_i$ for each $i\in\{1,2\}$.    When $E(f_i\mid\mathcal{Y}_\delta)=f_i$ for both $i=1$ and $i=2$, the claim is trivial, so we may further assume that for some $i\in\{1,2\}$, $E(f_i\mid\mathcal{Y}_\delta)=0$.  By Theorem \ref{wm} and Lemma \ref{tech:pair}, for any $\epsilon'>0$, we can find $K$ large enough so that every $\alpha^K$-sequence $s$ has an $n$ and an $\alpha$-subsequence $t$ such that 
  \begin{quote}
    for all $m\geq n$, both 
    \[\frac{1}{m}\sum_{i<m}\int\left[E(f_1T^{-i}g_1\mid\mathcal{Y}_\delta)-E(f_1\mid\mathcal{Y}_\delta)T^{-i} E(g_1\mid\mathcal{Y}_\delta)\right]^2d\mu<\epsilon\]
and
    \[\frac{1}{m}\sum_{i<m}\int\left[E(f_2T^{-i}g_2\mid\mathcal{Y}_\delta)-E(f_2\mid\mathcal{Y}_\delta)T^{-i} E(g_2\mid\mathcal{Y}_\delta)\right]^2d\mu<\epsilon\]

  \end{quote}
holds for $t$-many $\delta$.  But then, for such $\delta$ and $m\geq n$,
\[\frac{1}{m}\sum_{i<m}\int \left[E(f_1\otimes f_2(T^{-i}g_1\otimes T^{-i}g_2)\mid\mathcal{Y}_\delta)\right]^2d\mu\times_{\mathcal{Y}_\delta}\mu=\]
\[\frac{1}{m}\sum_{i< m}\int \left[E(f_1T^{-i}g_1\mid\mathcal{Y}_\delta)E(f_2T^{-i}g_2\mid\mathcal{Y}_\delta)\right]^2d\mu\]
is close to
\[\frac{1}{m}\sum_{i<m}\int \left[E(f_1\mid\mathcal{Y}_\delta)T^{-i}E(g_1\mid\mathcal{Y}_\delta)E(f_2\mid\mathcal{Y}_\delta)T^{-i}E(g_2\mid\mathcal{Y}_\delta)\right]^2d\mu\]
which is $0$ since either $E(f_1\mid\mathcal{Y}_\delta)=0$ or $E(f_2\mid\mathcal{Y}_\delta)=0$
\end{proof}

We need to generalize Theorem \ref{wm} to arbitrary polynomials of degree $1$.
\begin{lemma}
For every $\epsilon>0$ and $B>0$, there is a natural number $K$ such that for every $\Gamma\subseteq\mathbb{Z}^d$, every linearly independent $T_1,\ldots,T_t\in\mathbb{Z}^d\setminus\Gamma$, every tower $(\mathcal{Y}_\delta)_{\delta<\eta}$ of $\Gamma^+$-compact extensions, every $\alpha\geq\omega$, every $\alpha^K$-sequence $s$, every sequence of polynomials $p_j$ of degree $1$, and every $f,g\in L^\infty(Z_\Gamma(\mathcal{Y}_0))$ with $||f||_\infty,||g||_\infty\leq B$, there are an $n$ and an $\alpha$-subsequence $t$ of $s$ such that the property
  \begin{quote}
    for every $m\geq n$, $\frac{1}{m}\sum_{i<m}\int\left[E(f\vec T^p(i)g\mid\mathcal{Y}_\delta)-E(f\mid\mathcal{Y}_\delta)\vec T^p(i) E(g\mid\mathcal{Y}_\delta)\right]^2d\mu<\epsilon$\\
  \end{quote}
holds for $t$-many $\delta$.
\label{wmgen}
\end{lemma}
\begin{proof}
  Since $p$ has degree $1$, $\vec T^p(i)$ has the form $T_1^{c_1i}\cdots T_t^{c_ti}$ for some $c_1,\ldots,c_t$ not all $0$.  So we may apply Theorem \ref{wm} to $T_1^{c_1}\cdots T_t^{c_t}$.
\end{proof}
By the same argument, we obtain the analogous version of Theorem \ref{wm2}.

The following is our main theorem; we show that for any system $A$ of polynomials, there is a height such that all towers of that height contain many levels which are ``almost weak mixing'' along the system $A$.
\begin{theorem}
For every $\epsilon>0$, $B>0$, and $\gamma<\omega^\omega$, there are polynomials $q,q'$ such that for every integer $D$, there is an ordinal $\theta<\omega^{\gamma\cdot q(D)+q'(D)}$ such that for every $\Gamma\subseteq\mathbb{Z}^d$, every linearly independent $T_1,\ldots,T_t\in\mathbb{Z}^d\setminus\Gamma$, every tower $(\mathcal{Y}_\delta)_{\delta<\eta}$ of $\Gamma^+$-compact extensions, every system of essentially distinct non-constant integral polynomials $A$ with $|A|\leq D$, $o(A)\leq\gamma$, every $\{f_p\}_{p\in A}$ in $L^\infty(Z_\Gamma(\mathcal{Y}_0))$ with $||f_p||_\infty\leq B$ for each $p\in A$, every $\alpha\geq\omega$, every $\alpha^\theta$-sequence $s$, there are an $n$ and an $\alpha$-subsequence $t$ of $s$ such that
  \begin{quote}
    for every $m\geq n$, $||\frac{1}{m}\sum_{i<m}\left(\prod_{p\in A}\vec T^p(i)f_p - \prod_{p\in A}\vec T^p(i)E(f_p\mid\mathcal{Y}_\delta)\right)||<\epsilon$
  \end{quote}
holds for $t$-many $\delta$.
\label{polywm}
\end{theorem}
\begin{proof}
  %Write $\mathcal{Y}_\delta$ for $\mathcal{Y}^{k,\eta}_\delta(\mathcal{Z})$.
  We proceed by induction on $\gamma$.  Since $||\cdot||$ is subadditive, it suffices to split the key formula into $2^D$ cases, where each $f_p$ is replaced by either $E(f_p\mid\mathcal{Y}_\delta)$ or $f_p-E(f_p\mid\mathcal{Y}_\delta)$.  By Lemma \ref{tech:split}, the claim will follow from the claim for each of these simpler cases.  So it suffices to work with the property 
  \begin{quote}
    for every $m\geq n$, $||\frac{1}{m}\sum_{i<m}\left(\prod_{p\in A}\vec T^p(i)f'_p - \prod_{p\in A}\vec T^p(i)E(f'_p\mid\mathcal{Y}_\delta)\right)||<\epsilon$\label{eq:pwm}
  \end{quote}
where the $f'_p$ are terms of the form $E(f_p\mid\mathcal{Y}_\delta)$ or $f_p-E(f_p\mid\mathcal{Y}_\delta)$.

When every $f'_p$ is the term $E(f_p\mid\mathcal{Y}_\delta)$, this is clearly true for all $\delta$, so we may reduce to the property (**):
  \begin{quote}
    for every $m\geq n$, $||\frac{1}{m}\sum_{i<m}\prod_{p\in A}\vec T^p(i)f'_p||<\epsilon$\label{eq:pwms}
  \end{quote}
where $f'_{p_*}$ is $f_{p_*}-E(f_{p_*}\mid\mathcal{Y}_\delta)=0$ for some $p_*\in A$.

So let $\epsilon>0$, $B$, $\gamma$ be given.  Recall the set $A_h$ with $o(A_h)<o(A)$ (relative to any choice of $p_0\in\tilde A_h$ with $p_0$ of minimal weight); for $p\in A_h$, define $\tilde f_p$ to be one of $f'_p$, $\vec T^{p'}(h)f'_{p'}$, or $f'_{p'}\vec T^{p'}f_{p'}$ for some $p'\in A$ (the correct choice is directed by the calculations below).  There are two slightly different cases, depending on whether $p_*$ has degree $1$; if not, we may apply IH to obtain a $\theta$ so that, given $\{f_p\}$ and an $\alpha^\theta$-sequence $s$, there is a subsequence $t$, an $H$ large enough (see below), and an $N$ such that
\begin{quote}
for every $m\geq N$ and $h\in[-H,H]$, 
 \begin{equation*}
||\frac{1}{m}\sum_{i<m}\left(\prod_{p\in A_h}\vec T^p(n)\tilde f_p-\prod_{p\in A_h}\vec T^p(n)E(\tilde f_p\mid\mathcal{Y}_\delta)\right)||<\epsilon/2
\tag{$\dagger_1$}
\end{equation*}
\end{quote}
holds for $t$-many $\delta$, and in this case.  If $p_*$ does have degree $1$, we apply Lemma \ref{tech:tech} to Lemma \ref{wmgen} and IH, and obtain the property that for $t$-many $\delta$, $(\dagger_1)$ holds and additionally
\begin{quote}
  \begin{equation*}
\frac{1}{H}\sum_{h=1-H}^{H-1}\int\left[E(f_{p_*}\vec T^{p_*}(h)f_{p_*}\mid\mathcal{Y}_\delta)-E(f_{p_*}\mid\mathcal{Y}_\delta)\vec T^{p_*}(r)E(f_{p_*}\mid\mathcal{Y}_\delta)\right]^2 d\mu<\epsilon/2B^{4|A|}
\tag{$\dagger_2$}
\end{equation*}
\end{quote}
(The bounds from $1-H$ to $H-1$, instead of $0$ to $H-1$, are insignificant, since we can bound the positive and negative halves simultaneously using Lemma \ref{tech:split}, and so their sum is also bounded.)

We claim that whenever these conditions hold at $\delta$, (**) holds as well for $n:=\max\{N,cH\}$ where $c$ is an integer chosen large relative to $\epsilon$ and $B$.  To see this, observe that
\begin{eqnarray*}
  ||\frac{1}{m}\sum_{i<m}\prod_{p\in A}\vec T^p(i)f'_p||^2
&=||\frac{1}{m}\sum_{i<m}\frac{1}{H}\sum_{h<H}\prod_{p\in A}\vec T^p(i+h)f'_p||^2+\Psi'_H\\
&\leq\frac{1}{m}\sum_{i<m}\int\frac{1}{H^2}\sum_{h,h'<H}\prod_{p\in A}\vec T^p(i+h)f'_p \vec T^p(i+h')f'_p d\mu+\Psi'_H\\
&\leq\frac{1}{H}\sum_{h=1-H}^{H-1}\left(1-\frac{|h|}{H}\right)\frac{1}{m}\sum_{i<m}\int \prod_{p\in A}\vec T^p(i)f'_p \vec T^p(i+h)f'_p d\mu+\Psi''_H\\
\end{eqnarray*}
where $||\Psi'_H||$ and $||\Psi''_H||$ are small when $c$ is large and $m\geq n$.

Consider the expression $\vec T^p(i)f'_p\vec T^p(i+h)f'_p$.  If $p$ has degree $1$, this is $\vec T^p(i)(f'_p\cdot \vec T^p(h)f'_p)$, and if $p$ has degree greater than $1$, this is $\vec T^p(i)f'_p\cdot \vec T^p(i+h)(\vec T^{p}(h))^{-1}(\vec T^p(h)f'_p)$.  In particular,
\[\prod_{p\in A}\vec T^p(i)f'_p \vec T^p(i+h)f'_p =\prod_{p\in\tilde A_h}\vec T^p(i)\tilde f_p\]
where each $\tilde f_p$ is one of $f'_{p'}$, $\vec T^{p'}(h)f'_{p'}$, or $f'_{p'}\vec T^{p'}(h)f
_{p'}$ for some $p'\in A$.  For any $p_0\in\tilde A_h$, we also have
\[\int \prod_{p\in\tilde A_h}\vec T^p(i)\tilde f_p d\mu=\int \tilde f_{p_0}\prod_{p\in\tilde A_h\setminus\{p_0\}}\vec T^p(i)(\vec T^{p_0}(i))^{-1}\tilde f_p d\mu\]
and when $p_0$ has minimal weight, this is equal to
\[\int \tilde f_{p_0}\prod_{p\in A_h}\vec T^p(i)\tilde f_p d\mu.\]

When $m\geq n$, $(\dagger_1)$ implies that 
\[\frac{1}{m}\sum_{i<m}\int \tilde f_{p_0}\prod_{p\in A_h}\vec T^p(i)\tilde f_p d\mu\]
is close to
\[\frac{1}{m}\sum_{i<m}\int E(\tilde f_{p_0}\mid\mathcal{Y}_\delta)\prod_{p\in A_h}\vec T^p(i)E(\tilde f_p\mid\mathcal{Y}_\delta) d\mu.\]
Recall that $E(f'_{p_*}\mid\mathcal{Y}_\delta)=0$.  If $p_*$ has degree greater than $1$ then for all but an initial segment of the $h$, $\tilde f_p=f'_{p_*}$ for some $p\in A_h$, and therefore this expression is $0$.  If $p_*$ has degree equal to $1$, this expression is bounded by
\[||E(f'_{p_*} \vec T^{p_*}(h) f'_{p_*}\mid\mathcal{Y}_\delta)||_{L^2}\cdot\prod_{p\in A_h\setminus\{p_*\}}||f_p||_{L^\infty}^2.\]
But by $(\dagger_2)$,
\[\frac{1}{H}\sum_{h=1-H}^{H-1}||E(f'_{p_*} \vec T^{p_*}(h) f'_{p_*}\mid\mathcal{Y}_\delta)||_{L^2}\]
is close to $||E(f'_{p_*}\mid\mathcal{Y}_\delta)||^2_{L^2}=0$.
\end{proof}

Similarly, we can prove the same thing for formal elements of the relative square:
  \begin{theorem}
For every $\epsilon>0$, $B>0$, and $\gamma<\omega^\omega$, there are polynomials $q,q'$ such that for every integer $D$, there is an ordinal $\theta<\omega^{\gamma\cdot q(D)+q'(D)}$ such that for every $\Gamma\subseteq\mathbb{Z}^d$, every linearly independent $T_1,\ldots,T_t\in\mathbb{Z}^d\setminus\Gamma$, every tower $(\mathcal{Y}_\delta)_{\delta<\eta}$ of $\Gamma^+$-compact extensions, every system of essentially distinct non-constant integral polynomials $A$ with $o(A)\leq\gamma$ and $|A|\leq D$, and every $\{f_p\}$ in $L_0^\infty(Z_\Gamma(\mathcal{Y}_0)\times Z_\Gamma(\mathcal{Y}_0))$ with $||f_p||_\infty\leq B$, every $\alpha\geq\omega$, every $\alpha^\theta$-sequence $s$, there are an $n$ and an $\alpha$-subsequence $t$ of $s$ such that
  \begin{quote}
    for every $m\geq n$, $||\frac{1}{m}\sum_{i<m}\left(\prod_{p\in A}\vec T^p(i)f_p - \prod_{p\in A}\vec T^p(i)E(f_p\mid\mathcal{Y}_\delta)\right)||<\epsilon$
  \end{quote}
holds for $t$-many $\delta$.
\label{polywm2}
\end{theorem}
\begin{proof}
  Identical to the proof of the previous theorem, using Lemma \ref{wm2} in place of Theorem \ref{wm}.
\end{proof}

As in \cite{bergelson96}, we prove the following special cases:
\begin{corollary}
For every $\epsilon>0$, $B>0$, and $\gamma<\omega^\omega$, there is an ordinal $\theta<\omega^{\gamma\cdot\omega}$ such that for every $\Gamma\subseteq\mathbb{Z}^d$, every linearly independent $T_1,\ldots,T_t\in\mathbb{Z}^d\setminus\Gamma$, every tower $(\mathcal{Y}_\delta)_{\delta<\eta}$ of $\Gamma^+$-compact extensions, every system of pairwise essentially distinct polynomials $A$ with $o(A)\leq\gamma$, every $\{f_p\}_{p\in A}$ in $L^\infty(Z_\Gamma(\mathcal{Y}_0))$ with $||f_p||_\infty\leq B$ for each $p\in A$, every $\alpha\geq\omega$, every $\alpha^\theta$-sequence $s$, there are an $n$ and an $\alpha$-subsequence $t$ of $s$ such that
  \begin{quote}
    for every $m\geq n$, $\frac{1}{m}\sum_{i<m}\int\left|E(\prod_{p\in A}\vec T^p(i)f_p\mid\mathcal{Y}_\delta)-\prod_{p\in A}\vec T^p(i)E(f_p\mid\mathcal{Y}_\delta)\right|d\mu<\epsilon$
  \end{quote}
holds for $t$-many $\delta$.
\label{polywm2}
\end{corollary}
\begin{proof}
Repeatedly applying Lemma \ref{tech:split}, we may reduce to proving
  \begin{quote}
    for every $m\geq n$, $\frac{1}{m}\sum_{i<m}\int\left|E(\prod_{p\in A}\vec T^p(i)f'_p\mid\mathcal{Y}_\delta)\right|d\mu<\epsilon$
  \end{quote}
where $f'_p$ is either $f_p-E(f_p\mid\mathcal{Y}_\delta)$ or $E(f_p\mid\mathcal{Y}_\delta)$, and for at least one $p$ the first case holds.

By Theorem \ref{polywm2} and for suitable $\delta$, we may choose $\theta$, and given the remaining parameters (using $f_p\otimes f_p$ as the elements of $\mathcal{L}_0^\infty(Z_\Gamma(\mathcal{Y}_0)\times Z_\Gamma(\mathcal{Y}_0))$), since the elements of $A$ are pairwise essentially distinct and the $T_i$ are linearly independent, at most one $\vec T^p$ is constant.  Without loss of generality, we may assume this is a distinguished element $c\in A$.  Then we have
\begin{quote}
  for every $m\geq n$, $||\frac{1}{m}\sum_{i<m}\left(\prod_{p\in
      A\setminus\{c\}}(\vec T^p(i)\otimes \vec T^p(i))(f'_p\otimes f'_p) \right)||<\delta/B$
\end{quote}
for $t$-many $\delta$, and therefore
\[\frac{1}{m}\sum_{i<m}\int E(\prod_{p\in A\setminus\{c\}}(\vec T^p(i)\otimes \vec T^p(i))(f'_p\otimes f'_p)\mid\mathcal{Y}_\delta) d\mu<\delta.\]

But this implies that
\[\frac{1}{m}\sum_{i<m}\int f'_c E(\prod_{p\in A\setminus\{c\}}\vec T^p(i)f'_p\mid\mathcal{Y}_\delta)^2 d\mu<\delta\]
and, having chosen $\delta$ small in $\epsilon$ and $B$, also
\[\frac{1}{m}\sum_{i<m}\int \left|E(\prod_{p\in A}\vec T^p(i)f'_p\mid\mathcal{Y}_\delta)\right| d\mu<\epsilon.\]
\end{proof}

\begin{corollary}
 For every $\epsilon,\epsilon'>0$, and $\gamma<\omega^\omega$, there is an ordinal $\theta<\omega^{\gamma\cdot \omega}$ such that for every $\Gamma\subseteq\mathbb{Z}^d$, every linearly independent $T_1,\ldots,T_t\in\mathbb{Z}^d\setminus\Gamma$, every tower $(\mathcal{Y}_\delta)_{\delta<\eta}$ of $\Gamma^+$-compact extensions, every system of pairwise essentially distinct integral polynomials $A$ with $o(A)\leq\gamma$, every set $B$ measurable with respect to $Z_\Gamma(\mathcal{Y}_0)$, every $\alpha\geq\omega$, every $\alpha^\theta$-sequence $s$, there is an $\alpha$-subsequence $t$ of $s$ such that
  \begin{quote}
    $\mu\{x\mid \left|E(\bigcap_{p\in A}\vec T^p(i)B\mid\mathcal{Y}_\delta)(x)-\prod_{p\in A}E(\vec T^p(i)B\mid\mathcal{Y}_\delta)(x)\right|\geq\epsilon\}\geq\epsilon'$
    has density $0$
  \end{quote}
holds for $t$-many $\delta$.
\label{polywm3}
\end{corollary}
\begin{proof}
Apply Corollary \ref{polywm2} to $\chi_B$ obtain
\[\frac{1}{m}\sum_{i<m}|E(\bigcap_{p\in A}\vec T^p(i)B\mid\mathcal{Y}_\delta)(x)-\prod_{p\in A}E(\vec T^p(i)B\mid\mathcal{Y}_\delta)(x)|d\mu\geq\epsilon.\]
Setting $F_i(x):=E(\bigcap_{p\in A}\vec T^p(i)B\mid\mathcal{Y}_\delta)(x)-\prod_{p\in A}E(\vec T^p(i)B\mid\mathcal{Y}_\delta)(x)$, if the set
\[P:=\{n\mid \mu\{x\mid |F_n(x)|\geq\epsilon\}\geq\epsilon'\}\]
has positive density then we could find $c$ such that for arbitrarily large $N$, $\{n\in P\mid n<N\}\geq cN$, and therefore
\[\frac{1}{N}\sum_{i<N}\int |F_n(x)|d\mu\geq c\epsilon\epsilon',\]
which is a contradiction.
\end{proof}
We could insist on slightly more, requiring that the sets begin having low density at a fixed value $n$ simultaneously, for instance, but we do not need this.

\section{Almost Primitive Extensions}
The following is essentially shown in \cite{furstenberg79}:
\begin{lemma}
  If for every $\alpha$, $\mathcal{Y}_\alpha$ is an SZP-system then so is the system generated by $\bigcup_{\alpha<\gamma}\mathcal{Y}_\alpha$.
\label{SZPlimit}
\end{lemma}

Bergelson and Liebman \cite{bergelson96} prove the following:
\begin{lemma}
  Let $\mathcal{X}$ be an extension of $\mathcal{Y}$ with $\mathcal{Y}$ SZP, let $\mathbb{Z}^d=\Gamma\times\Delta$ so that $\mathcal{X}$ is compact relative to $\mathcal{Y}$ with respect to $\Gamma$, let $f\in L^2(\mathcal{X})$ be given.  Let $R_1,\ldots,R_r$ each have the form $\vec T^p$ for some sequence of pairwise distinct integral-zero polynomials $p_j$ and some $\vec T$ in $\Gamma$, and let $S_1,\ldots,S_s$ have the form $\vec T^p$ for some polynomial $p$ and some $\vec T$ in $\Delta$.  Let $B$ be measurable with respect to $\mathcal{Y}$ with $\mu(B)>0$, and let $\epsilon>0$.  Then there exist $P\subseteq\mathbb{N}$, $\underline{d}(P)>0$, a family of sets $\{B_n\mid n\in P\}$, each $B_n$ measurable with respect to $\mathcal{Y}$, and a $b>0$ so that, for any $n\in P$, $1\leq j\leq s$, $1\leq i\leq r$, we have
  \begin{itemize}
  \item $\mu(B_n)>b$
  \item $S_j(n)B_n\subseteq B$
  \item $\forall y\in B_n E((R_i(n)S_j(n)f-S_j(n)f)^2)(y)<\epsilon$.
  \end{itemize}
\end{lemma}

For our purposes, we need slightly more than this: we need a little bit of continuity allowing us to pass to an approximation to $\mathcal{Y}$ while retaining some control over the value of $b$.  Fortunately, the following strengthening follows immediately from their proof:
\begin{lemma}
  Let $\mathcal{X}$ be an extension of $\mathcal{Y}=\bigcup_{n<\omega}\mathcal{Y}_n$ with $\mathcal{Y}$ SZP, let $\mathbb{Z}^d=\Gamma\times\Delta$ so that $\mathcal{X}$ is compact relative to $\mathcal{Y}$ with respect to $\Gamma$, let $f\in L^2(\mathcal{X})$ be given.  Let $R_1,\ldots,R_r$ each have the form $\vec T^p$ for some polynomial $p$ and some $\vec T$ in $\Gamma$, and let $S_1,\ldots,S_s$ have the form $\vec T^p$ for some sequence of pairwise distinct integral-zero polynomials $p_j$ and some $\vec T$ in $\Delta$.  Let $B$ be measurable with respect to $\mathcal{Y}$ with $\mu(B)>0$, and let $\epsilon>0$.  There is a $b$, an $\epsilon'>0$, and an $M$ such that for each $m\geq M$ and each $B'$ measurable with respect to $\mathcal{Y}_m$ such that $\mu(B\bigtriangleup B')<\epsilon'$, there exists $P\subseteq\mathbb{N}$, $\underline{d}(P)>0$ and a family of sets $\{B_n\mid n\in P\}$, each $B_n$ measurable with respect to $\mathcal{Y}$, so that, for any $n\in P$, $1\leq j\leq s$, $1\leq i\leq r$, we have
  \begin{itemize}
  \item $\mu(B_n)>b$
  \item $S_j(n)B_n\subseteq B'$
  \item $\forall y\in B_n E((R_i(n)S_j(n)f-S_j(n)f)^2)(y)<\epsilon$.
  \end{itemize}
\end{lemma}
%Crucial idea: guarantee size of b' because \mu(B\bigtriangleup B') is small, so C_p is basically unchanged when $\delta<<b'$.  M and Q are independent of the factor, and b=b'/MQ

To simplify notation, we write $\omega_4$ for the ordinal $\omega^{\omega^{\omega^{\omega}}}$.
\begin{theorem}
If $(\mathcal{Y}_\delta)_{\delta<\omega_4}$ is a tower of $\Gamma$-compact extensions for every $\Gamma\subseteq\mathbb{Z}^d$ of dimension $\geq d-n$ and each $\mathcal{Y}_\delta$ is SZP then $Z_{d-n-1}(\mathcal{Y}_0)\cup\bigcup_{\delta<\omega_4}\mathcal{Y}_\delta$ is SZP.
\end{theorem}
\begin{proof}
By Lemma \ref{SZPlimit}, it suffices to show that $Z_{\Gamma}(\mathcal{Y}_0)\cup\mathcal{Y}_\delta$ is SZP for each $\Gamma\subseteq\mathbb{Z}^d$ of dimension $\geq d-n-1$ and $\delta<\omega_4$.  Further, since we can replace $\mathcal{Y}_0$ with $\mathcal{Y}_\delta$ and still have a tower of height $\omega_4$ above, it suffices to show that $Z_{\Gamma}(\mathcal{Y}_0)$ is SZP.

Fix some system of polynomials $A$ and some $\vec T_1,\ldots,T_t\in\mathbb{Z}^d$.  We may assume that for some $q$, $T_1,\ldots,T_q\in\Gamma$ and $T_{q+1},\ldots,T_t\not\in\Gamma$.

Let $C$ be a measurable set in $Z_{\Gamma}(\mathcal{Y}_0)$ with $\mu(C)>a$.  For each $p\in A$, we may write $\vec T^p(n)=R_p(n)S_p(n)$ where $R_p(n)=\prod_{j=1}^q T_j^{-p_j(n)}$ and $S_p(n)=\prod_{j=q+1}^t T_j^{-p_j(n)}$.  We may list the pairwise distinct components $\{R_1,\ldots,R_r\}$, $\{S_1,\ldots,S_s\}$ appearing.  It suffices to find a set $P$ of positive lower density and a $c>0$ so that for each $n\in P$, $\mu(\bigcap R_i(n)S_j(n) C)>c$.

By Lemma \ref{polywm3}, for each $m$, we may find a $\delta_m$ and a $P'_m\subseteq\mathbb{N}$ such that the complement of $P'_m$ has density $0$ and the set of $y$ such that
\[\left|E(\bigcap S_j(n)C\mid\mathcal{Y}_{\delta_m})(y)-\prod S_j E(C\mid\mathcal{Y}_{\delta_m})(y)\right|<a^s/2\]
has size less than $1/m$.  We may assume that the sequence $\delta_m$ is increasing, and set $\delta:=\lim_{m\rightarrow\infty}\delta_m$.

Let $\epsilon=\sqrt{a^s/16rs}$ and let $B:=\{y\mid E(C\mid\mathcal{Y}_{\delta}))(y)>a\}$; in particular, $\mu(B)>0$ and $B$ is measurable with respect to $\mathcal{Y}_{\delta}$.  Apply the preceeding lemma to $\chi_C$, $B$, and $\mathcal{Y}_\delta$; we obtain $b$ and an $\epsilon'$, and may choose $m$ so that $B':=\{y\mid E(C\mid\mathcal{Y}_{\delta_m}))(y)>a\}$ satisfies $\mu(B\bigtriangleup B')<\epsilon'$ and $1/m<b/2$ is sufficiently small relative to $b$.  We obtain $P,\{B_n\}$ such that, in particular, $E(R_i(n)S_j(n)C\bigtriangleup S_j(n)C\mid\mathcal{Y}_{\delta_m})(y)<4\epsilon^2$ and $E((S_j(n))^{-1}C\mid\mathcal{Y}_{\delta_m}))(y)>a$ for each $y\in B_n$.

Now consider some $n\in P'_m\cap P$; for all $y\in B_n$, $\prod S_j E(C\mid\mathcal{Y}_{\delta_m})(y)>a^s$, and for all $y\in B_n$ except for a set of size at most $b/2$,
\[\left|E(\bigcap S_j(n)C\mid\mathcal{Y}_{\delta_m})(y)-\prod S_j E(C\mid\mathcal{Y}_{\delta_m})(y)\right|<a^s/2.\]
Therefore, for such $y$, $E(\bigcap S_j(n)C\mid\mathcal{Y}_{\delta_m})(y)>a^s/2$.  Finally, for every $y\in B_n$, $E(R_i(n)S_j(n)C\bigtriangleup S_j(n)C\mid\mathcal{Y}_{\delta_m})(y)<4\epsilon^2$, and therefore $E(\bigcap R_i(n)S_j(n)C)(y)>a^s/4$.  Since this holds for a set of $y$ of size $\geq b/2$, it follows that $\mu(\bigcap R_i(n)S_j(n)C)\geq a^sb/8$.
\end{proof}

\begin{theorem}
  If $\mathcal{Z}$ is SZP then so is $\mathcal{Y}^{d-n,\omega_4}_\alpha(\mathcal{Z})$ for every $\alpha$.
\end{theorem}
\begin{proof}
  For all $\mathcal{Z}$ simultaneously, by main induction on $n$ and side induction on $\alpha$.  For all $n$, the limit case follows from Lemma \ref{SZPlimit}.

When $n=0$, the claim follows at successor stages since $\mathcal{Y}^{d,\omega_4}_{\alpha+1}(\mathcal{Z})$ is a compact extension of $\mathcal{Y}^{d,\omega_4}_{\alpha}(\mathcal{Z})$ and it is shown in \cite{bergelson96} that compact extensions preserve the SZP property.

If the claim holds for $n$ and $\mathcal{Y}^{d-n-1,\omega_4}_{\alpha}(\mathcal{Z})$ is SZP, the previous theorem implies that
\[\mathcal{Y}^{d-n-1,\omega_4}_{\alpha+1}(\mathcal{Z})=Z_{d-n-1}(\mathcal{Y}^{d-n-1,\omega_4}_\alpha(\mathcal{Z}))\cup\mathcal{Y}^{d-n,\omega_4}_{\omega_4}(\mathcal{Y}^{d-n-1,\omega_4}_\alpha(\mathcal{Z}))\]
is compact since, by main IH, $\mathcal{Y}^{d-n,\omega_4}_{\omega_4}(\mathcal{Y}^{d-n-1,\omega_4}_\alpha(\mathcal{Z}))$ is SZP.
\end{proof}

In particular, the trivial factor is SZP, so if $\mathcal{Z}$ is the trivial factor, the previous theorem implies that the entire space $\mathcal{X}=\mathcal{Y}^{0,\omega_4}_1(\mathcal{Z})$ is SZP as well.  Therefore all dynamical systems are SZP.

\bibliographystyle{plain}
\bibliography{../Bibliographies/ergodic}
\end{document}